\documentclass[12pt]{amsart}
\usepackage{amscd}
\usepackage{amssymb,amsmath,amsthm}
\theoremstyle{plain}
\newtheorem{thm}{Theorem}[section]
\newtheorem{lem}[thm]{Lemma}
\newtheorem{cor}[thm]{Corollary}
\newtheorem{defn-lem}[thm]{Definition-Lemma}

\newtheorem{prop}[thm]{Proposition}

\theoremstyle{definition}
\newtheorem{defn}[thm]{Definition}

\newtheorem{rem}[thm]{Remark}

\def\md #1#2#3#4#5 {\left(
                        \begin{matrix}
             #1 & #2 \\
             #3 & #4
                        \end{matrix}
                      \right)- #5}

\def\ma #1#2#3#4 {\left(
                        \begin{matrix}
             #1 & #2 \\
             #3 & #4
                        \end{matrix}
                      \right)}
\def \mu  {\mathcal{M}(C(X)\otimes B)}
\def \mt  {\mathcal{M}(B)}
\def\Ind{\operatorname{Ind}}
\def\Im{\operatorname{Im}}

\def \Ker {\operatorname{Ker}}

\def\KK{\operatorname{KK}}
\newcommand{\mc}{\mathcal}
\begin{document}
\title [Homotopy classification of projections in the corona algebra]
       {Homotopy classification of homogeneous projections in the corona algebra of a non-simple $C\sp*$-algebra}

\begin{abstract}
In this paper  we consider certain proejctions in the corona algebra of  $C(X)\otimes B$ associated to $(p_0, p_1, \dots, p_n)$ where $p_i: X_i \to \mt_s$ a continuous projection valued section to the multiplier algebra of a stable $C\sp*$-algebra $B$ for each i. Here $X_i$'s are closed intervals given by a partition $\{x_1, \dots, x_n\}$ on the interior of $X$ and adjacent sections differ by compacts at each partition point.  Assuming a kind of homogeneity on the projection we characterize  when two such projections are homotopy equivalent.
\end{abstract}

\author { Hyun Ho \quad Lee }

\address {Department of Mathematics\\
          University of Ulsan\\
         Ulsan, South Korea 680-749 }
\email{hadamard@ulsan.ac.kr}


\date{}
\thanks{}
\maketitle

\section{Introduction}

In this article, we study the projections in the corona algebra of $C(X)\otimes B$ where $X$ is a locally compact, Hausdorff space of finite (covering) dimension and $B$ a simple, stable $C\sp*$-algebra such that the multiplier algebra $\mt$ is of real rank zero.  In \cite{BL}, for the case that $X$ are one dimensional connected spaces and $B=K$ the algebra of compact operators on a separable Hilbert space we gave  necessary and sufficient conditions for two projections in the corona algebra to be Muarry-von Neumann equivalent, unitarily equivalent, homotopy equivalent. The works to generalize and extend the results of \cite{BL} were carried out in \cite{Lee}, \cite{Lee2}; BDF's classical definition was extended to an arbitrary $C\sp*$-algebra and  the important question about lifting a projection from the corona algebra was solved for $X$ one dimensional compact space. However, the equivalence relations of two projections have been remaind unsolved.  Let us briefly explain why this problem is interesting; if we consider the extensions defined by two projections $\mathbf{p},\mathbf{q}$, say $\tau_{\mathbf{p}}, \tau_{\mathbf{q}}$ respectively, the homotopy(unitary) equivalence of $\mathbf{p}, \mathbf{q}$ implies the strong(weak) unitary equivalence of $\tau_{\mathbf{p}}, \tau_{\mathbf{q}}$.  Thus  our aim of this article is to obtain charaterizations of fundamental equivalence relations of two (homogeneous) projections in the corona algebra of $C(X)\otimes B$.         

From now on $X$ is $[0,1],(-\infty,\infty),[0,\infty)$ or the circle $[0,1]/\{0,1\}$. 
When $X$ is compact, let $I=C(X)\otimes B$ which is the $C\sp{*}$-algebra of (norm continuous) functions from $X$ to $B$. When $X$ is not compact,  let $I=C_0(X)\otimes B$ which is the $C\sp*$-algebra of continuous functions from $X$ to $B$ vanishing at infinity. Then $\mu$ is $C_b(X, M(B)_s)$, which is the space of
   bounded functions from $X$ to $\mt$, where $\mt$ is given the strict
   topology \cite{ATP}. Let $\mathcal{Q}(I)=\mathcal{M}(I)/I$ be the corona algebra of $I$ and also let $\pi:\mathcal{M}(I) \to \mathcal{Q}(I)$ be the natural quotient map. Then an element
   $\mathbf{f}$ of the corona algebra can be represented as follows:  Consider a
 finite partition of $X$, or $X \smallsetminus \{0,1\}$ when $X=\mathbb{T}$ given by partition points $x_1 < x_2 < \cdots
 < x_n $ all of which are in the interior of $X$ and divide $X$ into
 $n+1$ (closed) subintervals $X_0,X_1,\cdots,X_{n}$. We can take $f_i \in
 C_b(X_i, M(B)_s)$ such that $f_i(x_i) -f_{i-1}(x_i)\in B$
 for $i=1,2,\cdots,n$ and $f_0(x_0)-f_n(x_0) \in B$  where $x_0=0=1$ if $X$ is the circle.
 \begin{lem}\cite[Lemma 3.1]{Lee}
 The coset in $\mathcal{Q}(I)$ represented by
 $(f_0,\cdots,f_n)$  consists of functions $f$ in $M(I)$ such that $f- f_i \in
 C(X_i)\otimes B$ for every $i$ and $f-f_i $ vanishes (in norm) at any
 infinite end point of $X_i$.
 \end{lem}
 Thus $(f_0,\cdots,f_n)$ and $(g_0,\cdots,g_n)$
 define the same element of $\mathcal{C}(I)$ if and only if $f_i - g_i \in
 C(X_i)\otimes B$ for $i=0,\cdots,n$ if $X$ is compact.  $(f_0,\cdots,f_n)$ and $(g_0,\cdots,g_n)$
 define the same element of $\mathcal{C}(I)$ if and only if  $f_i - g_i \in
 C(X_i)\otimes B$ for $i=0,\cdots,n-1$, $f_n -g_n \in
 C_0([x_n,\infty))\otimes B$  if $X$ is $[0.\infty)$.  $(f_0,\cdots,f_n)$ and $(g_0,\cdots,g_n)$
 define the same element of $\mathcal{C}(I)$ if and only if
  $f_i - g_i \in
 C(X_i)\otimes B$ for $i=1,\cdots,n-1$, $f_n -g_n \in
 C_0([x_n,\infty))\otimes B$, $f_0-g_0 \in
 C_0((-\infty,x_1])\otimes B$ if $X=(-\infty,\infty)$.\\
 
 The following theorem shows why the above description of the element in the corona algebra has an advantage; the projections in the corona algebra are ``locally liftable''.
\begin{thm}\cite[Theorem 3.2]{Lee}\label{T:locallift}
Let $I$ be $C(X)\otimes B$ or $C_0(X)\otimes B$ where $B$ is a stable $C\sp{*}$-algebra such that $\mt$  has real rank zero.
Then a projection $\mathbf{p}$ in $\mc{Q}(I)$ can be represented by $(p_0,p_1,\cdots, p_n)$  under a suitable partition $\{x_1, \dots, x_n\}$ of $X$ such that $p_i$ is a projection valued function in $C(X_i)\otimes \mt_s$ and $p_i(x_i)-p_{i-1}(x_i) \in B$ for each $i=1,\dots, n$. We call $(p_0,p_1,\cdots, p_n)$ the local representation or lifting for $\mathbf{p}$.
\end{thm}
    
Recall that a closed submodule $E$ of the Hilbert module $F$ over $B$ is complementable if and only if there is a submodule $G$ orthogonal to $E$ such that $E\oplus G =F$. The Kasparov stabilization theorem says that a countably generated closed submodule of the standard Hilbert module $H_B$ is complementable whence it is the image of a projection in $\mathcal{L}(H_B)$ \cite{Kas80}. Let $\mathfrak{F}=((F_x)_{\{x\in X\}},\Gamma)$ be a continuous field of Hilbert modules over a locally compact Hausdorf space $X$ in the Dixmier-Duady sense.  A continuous field of Hilbert modules $((E_x)_{\{x\in X\}},\Gamma')$ is said to be complementable to $\mathfrak{F}$ when $E_x$ is a complementable submodule of $F_x$ for each $x \in X$. Then the following is the natural correspondence between the continuous field and its section map.  
    \begin{prop}\cite[Proposition 3.7]{Lee2}\label{P:bundlesection}
 A complementable subfield $((E_x)_{\{x\in X\}},\Gamma')$ of the constant module $((H_B)_{\{x\in X\}}, \Gamma)$, where $\Gamma$ consists of the (norm) continuous section from $X$ to $H_B$, is in one to one correspondence to a continuous projection-valued map $p: X \to \mathcal{L}(H_B)$, where the latter is equipped with the $*$-strong topology. (In general, we say that $\{T_i\}$ in $\mathcal{L}(X)$ converges to $T$ $*$-strongly if and only if both $T_i(x) \to T(x)$ and $T_i^*(x) \to T^*(x)$ in $X$ for all $x\in X$.)
 \end{prop}

 Combining  Proposition \ref{P:bundlesection} with Threorem \ref{T:locallift}, we can view a projection $\mathbf{p} \in \mathcal{Q}(I)$ as the complementable Hilbert module bundle by the local lifting  $(p_0,p_1, \dots, p_n )$. This point of view is a key to study a projection in the corona algebra of $C(X)\otimes B$ and deform it.   For instance, we are going to focus on more tractable projections, so called ``homogeneous'' ones. This means that the extension generated by the proejection $\mathbf{p}$ is homogeneous in the sense of Pimsner, Popa, Voiculescu \cite{PPV} (Pimsner, Popa, Voiculescu introduced this notion a quater centuary ago). Equally we prefer to say that $p_i(x)$'s are full and properly inifinite for all $i$ and $x\in X_i$ or the fibers are isomorphic to $H_B$. Thus we use this algebraic condition on fibers as our homogeneity condition on $\mathbf{p}$ from now on.\\

 Passing to the results of this article, we note that much of the material in this note is about adaptations of many techniques  from the compact case to more general $C\sp*$-algebras. The   key result is Proposition 3.6.

\section{Preliminaries; K-theoretic data associated with the bundle}
To help the reader we collect some results in this section. Thus this section  does not contain new results except Lemma \ref{L:BDF}.  To simplify the notation we assume that $B$ is stable. Accordingly we indentify the multiplier algebra $\mt$ with $\mathcal{L}(H_B)$ \cite{Kas80}.
Since we deal with the projection represneted by a local lifting $(p_0,\dots,p_n)$ whose discontinuity conditions are differences of two projections, it is desirable to quantify these discontinuities. The next notion provides the quantification of such discontinuities using K-theory. We refer the reader to \cite{Cu,Kas81} for rudiments of KK-theory.  
\begin{defn}\cite[Definition 2.1]{Lee2}\label{D:BDF}
Given two projections $p,q \in \mt$ such that $p-q \in B$, we consider representations $\phi,\psi $ from $\mathbb{C}$ to $\mt$ such that $\phi(1)=p,\psi(1)=q$. Then $(\phi,\psi)$ is a Cuntz pair so that
we define $[p:q]$ as the class $[\phi,\psi]\in \KK_h(\mathbb{C},B) \simeq K(B)$ and call the (generalized) essential codimension of $p$ and $q$.
\end{defn}
The following properties of $[\,:\,]$ are natrually expected to behave like the original essential codimension (see \cite[section 2]{B}). 
\begin{lem}\cite[Lemma 2.3]{Lee2}\label{L:properties}
$[\,:\,]$ has the following properties.
\begin{enumerate}
\item $[p_1:p_2]=[p_1]_0-[p_2]_0$ if either $p_1$ or $p_2$ belongs to $B$, where $[p_i]_0$ is the $K_0$-class of a projection $p_i$,
 \item $[p_1:p_2]=-[p_2:p_1],$
 \item $[p_1:p_3]=[p_1:p_2]+[p_2:p_3],\text{when sensible},$
 \item $[p_1+q_1:p_2+q_2]=[p_1:p_2]+[q_1:q_2], \, \text{when sensible.}$
\end{enumerate}
\end{lem}
 \begin{lem}\cite[Lemma 2.4]{Lee2}\label{L:unitaryequi}
  Let $p$ and $q$ be projections in $\mt$ such that $p-q \in B$. If there is a
  unitary $U \in 1+ B$ such that $UpU^*=q$, then $[p:q]=0$. In
  particular, if $\| p- q\|<1 $, then $[p:q]=0$.
 \end{lem}
  A projection  $p$ in a unital $C\sp*$-algebra $A$ is called a \emph{halving} projection if both $1-p$ and $p$ are Murray-von Neumann equivalent to the unit in $A$ and a projection $p$ in $A$ is called properly infinite if there are mutually orthogonal projections $e$, $f$ in $A$ such that $e \le p$, $f \le p$, and $e \sim f \sim p$. Then it is easy to check that a projection in the multiplier algebra of a stable $C\sp*$-aglebra is Murray-von Neumann equivalent to $1$ if and only if it is full and properly infinite.

The following lemma shows that the above definition of essential codimension is the right generalization of classical BDF's essential codimension. Recall that BDF's original definition of the essential codimension of $p$ and $q$ in $B(H)$ is given by the Fredholm index of $V^*W$ where $V$ and $W$ are  isometries such that $VV^*=q$ and $W^*W=p$ \cite{BDF}.  An operator on $H_B$ is called a Fredholom operator when it is invertible modulo $\mathcal{K}(H_B)$ the ideal of compact operators. In fact, a generalized Atkinson theorem says that an opertor $F$ for which there exists a compact $K\in \mathcal{K}(H_B)$ such that $\Ker (F+K)$ and $\Ker (F+K)\sp*$ finitely generated and $\Im(F+K)$ closed is a Fredholm operator and vice virsa  \cite {Mi}. Thus we can define an index of a Fredholm operator in $K_0(B)$ as the diffence of two classes of finitely generated modules. Let us denote its index by $\Ind$. For more details, we refer the reader to \cite{We,Mi}.
\begin{lem}\label{L:BDF}
Suppose $p$, $q$ are full properly infinite projections in $\mt$ such that $p-q\in B$. Then $[p:q]=\Ind(V^*W)$ where $V, W$ isometries whose range projections are $q, p$ respectively.
\end{lem}
\begin{proof}
Let $W$ be an isometry in $\mt$ such that $WW^*=p$, $V$ isometry in $\mt$ such that $VV^*=q$. Then $V^*W$ is a unitary modulo $B$. Thus $\Ind(V^*W)$ is well defined. Let $\phi_i:\mathbb{C} \to \mt$ such that $\phi_1(1)=p$, $\phi_2(1)=q$.
Then a triple \[(\phi_1,\phi_2,1)=\left(H_B \oplus H_B, \left(\begin{array}{cc}
                                                                            \phi_1 &  \\
                                                                             & \phi_2 \\
                                                                          \end{array}
                                                                        \right), \left(
                                                                                   \begin{array}{cc}
                                                                                     0 & 1 \\
                                                                                     1 & 0 \\
                                                                                   \end{array}
                                                                                 \right)
 \right)\] is a compact perturbation of \[\left( H_B \oplus H_B, \left(\begin{array}{cc}
                                                                            p &  \\
                                                                             & q \\
                                                                          \end{array}
                                                                        \right), \left(
                                                                                   \begin{array}{cc}
                                                                                     0 & pq \\
                                                                                     qp & 0 \\
                                                                                   \end{array}
                                                                                 \right)
 \right). \]
 The latter is decomposed to $((1-p)(H_B) \oplus( 1-q)(H_B),0,0)\\ \oplus \left(p(H_B) \oplus q(H_B), \left(\begin{array}{cc}
                                                                            p &  \\
                                                                             & q \\
                                                                          \end{array}
                                                                        \right), \left(
                                                                                   \begin{array}{cc}
                                                                                     0 & pq \\
                                                                                     qp & 0 \\
                                                                                   \end{array}
                                                                                 \right)
 \right)$ so that
 $(\phi_1,\phi_2,1)$ is represented as \[\left(p(H_B) \oplus q(H_B), \left(\begin{array}{cc}
                                                                            p &  \\
                                                                             & q \\
                                                                          \end{array}
                                                                        \right), \left(
                                                                                   \begin{array}{cc}
                                                                                     0 & pq \\
                                                                                     qp & 0 \\
                                                                                   \end{array}
                                                                                 \right)
 \right).\] By the isomorphism $\Ind: \KK(\mathbb{C}, B) \to K_0(B)$ which sends $(\phi_1, \phi_2,1)$ to $\Ind(pq)$, $[p:q]=\Ind(pq)=\Ind(V^*qpW)=\Ind(V^*W)$.
 In fact, it is realized as the following diagram
\[
 \begin{CD}
  H_B\oplus H_B @>V^*qpW\oplus I>>  H_B\oplus H_B\\
  @VV W\oplus I V                @A V^*\oplus I  AA \\
  p(H_B)\oplus H_B         @>> qp\oplus I > q(H_B)\oplus H_B.
  \end{CD}
 \]
\end{proof}

The following interesting fact for a section $p\in \mu$ was proved by Kucerovsky and Ng \cite[Theorem 3.1]{KuNg} under the condition that its pointwise evaluation $p_x$ is halving for each $x\in X$. However, it can be shown under a weaker condtion that $p_x$ is full and properly infinite for each $x$.    
\begin{thm}\cite[Corollary 3.8]{Lee2}\label{T:triviality}
Let $B$ be a separable, $\sigma_p$-unital, and stable $C\sp*$-algebra so that $\mt$ has a halving projection and X a finite dimensional compact Hausdorff space. Then a complementable subfield of Hilbert modules associated with a projection-valued map $p:X \to \mt_s$ is isomorphic to a trivial field provided that any pointwise evaluation $p_x$ is full and properly infinite in $\mt$.
\end{thm}
This theorem means that $p\in \mu$ is Murray-von Neumann equivalent to $1_{\mu}$ if $p_x$ is Murray-von Neumann equivalent to $1_{\mt}$. Theorem \ref{T:triviality} will be used repeatedly in the next section. \\
 
The following lemma, which was shown in \cite{Lee2}, says that if a projection valued section or the corresponding Hilbert module bundle has ambient fibers or ``infinite dimensional'' fibers it has any ``finite dimensional'' trivial subbundle. It means that we can deform an ambient Hilbert module bundle by disembedding or embedding a subbundle. Again the infinite dimensional fiber condition is expressed in terms of an algebraic condition on  each pointwise evaluation of the section; i.e., the projection is Murray-von Neumann equivalent to $1_{\mt}$.        
\begin{lem}\cite[Lemma 3.10]{Lee2}\label{L:subprojection}
Let $p$ be a section from $X$ to $\mt$ with respect to the strict topology on $\mt$  where  $B$ is a $\sigma_p$-unital, stable $C\sp*$-algebra of real rank zero such that $\mt$ contains a halving full projection. In addition, when we denote its image on $x\in X$ by $p_x$, assume that $p_x$ is a properly infinite full projection for each $x$. Then for any $\alpha \in K_0(B)$ there exists a norm continuous section $\mathbf{r}$ from $X$ to $B$ such that  $r_x \leq p_x$ such that $[r]_{K_0(B)}=\alpha$.
\end{lem}

We will call a $B$-valued norm continuous section  simply a `` finite '' section map.    Since two projections are same in the corona algebra  when they are equal modulo any $B$-valued norm continuous section map, Lemma \ref{L:subprojection} enable us to deform the local lifting without changing $\mathbf{p}$ in the corona algebra. In other words, we allow the addition or subtraction of any $B$-valued section map in each interval $X_i$ while keeping the same class. 

\section{Main Results; Classification}

Let $\mathbf{p}=(p_0,p_1,\dots,p_n), \mathbf{q}=(q_0,q_1,\dots,q_n)$ be two projections in the corona algebra of $C(X)\otimes B$ for a fixed partition of $X$.  In addition, we put an algebraic condition on each fiber such that the fiber is isomophic to $H_B$. Since $p_i(x)$ is a projection in $\mt$ for any $x$ in $X_i$, this condition is equivalent to the fact that $p_i(x)$ is full and properly infinite. If $B=K$, then this condition means that each fiber is an infinite dimensional subspace of a separable infinite dimensional Hilbert space $H$. From discontinuity condtions given by the partition we can associate to these projections the essential codimensions $k_i=[p_i(x_i):p_{i-1}(x_i)], l_i=[q_i(x_i),q_{i-1}(x_i)] \in K_0(B)$ for $i=1,\dots, n$.

 Without discontinuity conditions, it has been shown that if the bundle associated to $p\in \mu$ has the fiber isomorphic to $H_B$ for any $x \in X$  the bunlde is trivial for a general stable $C\sp*$-algebra $B$ in \cite[Corollary 3.8]{Lee2}. When $X$ has no endpoints, then we can show that finite discontinuities of any homogeneous Hilbert module bundle for a smaller class of $C\sp*$-algebras $B$ can be resovled. Thus this bundle is isomorphic to the trivial bundle so that the corrsponding projection $\mathbf{p}$, which is liftable, is Murray-von Neumann equivalent $1_{\mathcal{Q}(C(X)\otimes B)}$. So does $\mathbf{q}$. Hence we conclude that $\mathbf{p}\sim \mathbf{q} \sim 1_{\mathcal{Q}(C(X)\otimes B)}$. 
\begin{thm}
 Let $B$ be a $\sigma$-unital, purely infinite, simple $C\sp*$-algebra such that $\mt$ has real rank zero. Suppose $\mathbf{p}=(p_0,p_1,\dots,p_n)$ be the projection in the corona algebra of $C(X)\otimes B$ such that $p_i(x)$ is full and properly infinite projection in $\mt$ for all $x\in X_i$ and $i$. If $X$ has no endpoints, then $\mathbf{p}\sim 1_{\mathcal{Q}(C(X)\otimes B)}$. 
\end{thm}
\begin{proof}
Take $p'_0=p_0$. By Lemma \ref{L:subprojection} we can take a norm continuous projection valued function $r_1:X_1 \to B$ such that $[r_1]_0=k_1$ and $r_1 \le p_1$. Now we consider $p'_1=p_1-r_1$. Note that $[p'_1(x_1):p_{0}(x_1)]=[p_1(x_1):p_{0}(x_1)]-[r_1(x_1):0]=[p_1(x_1):p_{0}(x_1)]-[r_1]_0=0$ by Lemma \ref{L:properties}. Then we replace $p_1$ by $p_1-r_1$. Note that $(p'_0,p'_1,p_2,\dots,p_n)$ still represents the same $\mathbf{p}$. But $[p_2(x_2):p'_1(x_2))]=[p_2(x_2):p_1(x_2)]-[0:r_1(x_2)]=[p_2(x_2):p_1(x_2)]+[r_1(x_2):0]=k_2+k_1$
Thus in this way we can obtain $(p'_0,p'_1,\dots,p_n')$ inductively which define the same element $\mathbf{p}$ such that $[p'_i(x_i):p'_{i-1}(x_i)]=0$. Then as in the proof of Theorem 3.3 in \cite{Lee} we can glue the adjoining sections at the partition points so that we have a continuous bundle. Then by Theorem \ref{T:triviality} this bundle is trivial. 
\end{proof}
\begin{rem}
In fact, there exist $p$ and $q$ the lifts of $\mathbf{p}$ and $\mathbf{q}$ in the $\mu$  respectively. Hence the bundles are isomorphic in a more strong sense. i.e., $p \sim q \sim 1_{\mu}$.
\end{rem}
Now we consider the general case. Suppose $\mathbf{p} \sim  \mathbf{q}$.
Then there exists an element $u$ in $\mu $ such that $\pi(u)\sp*\pi(u)=\mathbf{p}$ and $\pi(u)\pi(u)\sp*=\mathbf{q}$. Let $u|_{X_i}$ be the restriction of $u$ on $X_i$ and $u_i=q_iu|_{X_i}p_i$. Then it follows that $u_i(x_i)-u_{i-1}(x_i)\in B$.
Since we assume that $p_i(x) \sim 1_{\mt}$ for each $x \in X_i$, there exists $v_i: X_i \to \mt $ such that $v_i\sp* v_i=1_{\mathcal{M}(C(X_i)\otimes B)}$ and $v_iv_i\sp*=p_i$ by Theorem \ref{T:triviality}.
Similarly, there exists $w_i: X_i \to \mt $ such that $w_i w_i\sp*=1_{\mathcal{M}(C(X_i)\otimes B)}$ and $w_i\sp*w_i=q_i$.
Then we can check that $w_iu_iv_i$ is a unitary in the corona algebra of $C(X_i)\otimes B$.
Thus we have a partial isometry $\widetilde{u_i}$ in $\mathcal{M}(C(X_i)\otimes B)$ such that $\pi(\widetilde{u_i}^*\widetilde{u_i})=\pi(\widetilde{u_i}\widetilde{u_i}^*)=1_{\mathcal{Q}(C(X_i)\otimes B)}$.
\begin{prop}
$T_i=w_i^*\widetilde{u_i}v_i^*$ is also a partial isometry.
\end{prop}
\begin{proof}
It is enough to show that $T_iT_i^*$ is a projection.
Since it is obviously self-adjoint, the proof is complete by
\[\begin{split}
T_iT_i^*T_iT_i^*&=w_i^*\widetilde{u_i}\widetilde{u_i}^*w_iw_i^*\widetilde{u_i}\widetilde{u_i}^*w_i\\
                &=w_i^*\widetilde{u_i}\widetilde{u_i}^*\widetilde{u_i}\widetilde{u_i}^*w_i\\
                &=w_i^*\widetilde{u_i}\widetilde{u_i}^*w_i=T_iT_i^*.
\end{split}\]
\end{proof}
Note that $T_iT_i^*\leq w^*w=q_i$, $T_i^*T_i\leq v_iv_i^*=p_i$ and  
\[  \pi(T_i^*T_i)=\pi(v_iv_i^*)=\pi(p_i), \quad \pi(T_iT_i^*)=\pi(w_i^*w_i)=\pi(q_i).\]
In other words $p_i-T_i^*T_i \in C(X_i)\otimes B$, $q_i-T_iT_i^* \in C(X_i)\otimes B$. Thus we can define an index $t_i \in K_0(B)$ as $[p_i-T_i^*T_i]_0-[q_i-T_iT_i^*]_0$ where $[\,]_0$ means an element of $K_0(B)$. If $X_i$ contains an infinite end point, then $t_i=0$ by Lemma \ref{L:unitaryequi}.\\

If a projection in $B(H)$ satisfies the condition $1-p\in K$, then $1-p$ has a finite rank, so $p$ has an infinite dimensional range. Thus it is Murry-von Neumann equivalent to $1_{B(H)}$. We have the following analogue of this fact on $H_B$.  
\begin{lem}\label{L:Wegge-Olsen}
Let $p\in \mt$ such that $1-p \in B$, then $p \sim 1_{\mt}$ in $\mt$.
\end{lem}
\begin{proof}Since it is well known, we omit the proof. See \cite{We}.\end{proof}
\begin{lem}\label{L:index-essentialcodimension}
Consider two projections $p,q$ in $\mt$ such that $q \le p$, $p-q \in B$, and $p \sim 1_{\mt}$.
Then $[p:q]=[p-q]_0$.
\end{lem}

\begin{proof}
First we claim that there is an isometry whose image projection is $q$. By the assumption on $p$ we let $x$ be an isometry whose image projection is $p$.
Consider $p'=x^*px$ and $q'=x^*qx$, then $p' \ge q'$ and $p'-q'=1-q' \in B$. Then by Lemma \ref{L:Wegge-Olsen} $q' \sim 1_{\mt}$.  But $(qx)(qx)^*=qxx^*q=qpq=q \sim x^*q^*qx=x^*qx \sim 1_{\mt}$. So the claim is complete.

Now we take $W,V$ isometries corresponding to $p,q$ respectively. Then $V^*W(V^*W)^*=V^*WW^*V=V^*pV=V^*qpV=V^*qV=V^*VV^*V=1$.
Thus \[\begin{split}
\Ind(V^*W)&=[1-(V^*W)^*V^*W]_0=[1-W^*VV^*W]_0=[W^*(p-VV^*)W]_0\\
          &=[W^*(p-q)W]_0=[(p-q)WW^*(p-q)]_0=[(p-q)p(p-q)]_0\\
          &=[p-q]_0.
\end{split}\]
\end{proof}
The following formula is one of key results in this article.
\begin{prop}\label{P:index-essentialcodimension}
Assume $\mathbf{p}=(p_0,p_1,\dots,p_n) \sim \mathbf{q}=(q_0,q_1,\dots,q_n)$. Then the index $t_i=[p_i-T_i^*T_i]_0 - [q_i-T_iT_i^*]_0$ can be written using the essential codimension as follows;
\[t_i=[p_i:T_i^*T_i]-[q_i:T_iT_i^*].\]
\end{prop}
\begin{proof}
Combine Lemma \ref{L:index-essentialcodimension} with the definition of the index $t_i$. 
\end{proof}
Now let us proceed to obtain the sufficient conditions for Murray von Neumann equivalence of $\mathbf{p,q}$. Note that $p_i(x)-T_i^*(x)T_i(x)\in B$, $q_i(x)-T_i(x)T_i^*(x)\in B$ for each $x \in X_i$ imply that $T_i^*T_i(x_i)-T_{i-1}^*(x_i)T_{i-1}(x_i)\in B$ and $T_iT_i^*(x_i)-T_{i-1}(x_i)T_{i-1}^*(x_i)\in B$. Combining this fact with Lemma \ref{L:properties}, we get the following.

\begin{equation} \label{E:(1)}
\begin{split}
[p_i(x_i):p_{i-1}(x_i)]&=[p_i(x_i):T_i^*(x_i)T_i(x_i)]+[T_i^*(x_i)T_i(x_i):T_{i-1}^*(x_i)T_{i-1}(x_i)]\\
                       &+[T_{i-1}^*(x_i)T_{i-1}(x_i):p_{i-1}(x_i)]
\end{split}
\end{equation}

\begin{equation} \label{E:(2)}
\begin{split}
[q_i(x_i):q_{i-1}(x_i)]&=[q_i(x_i):T_i(x_i)T_i^*(x_i)]+[T_i(x_i)T_i^*(x_i):T_{i-1}(x_i)T_{i-1}^*(x_i)]\\
                       &+[T_{i-1}(x_i)T_{i-1}^*(x_i):q_{i-1}(x_i)]
\end{split}
\end{equation}
Recall that $T_i=w_i^*\widetilde{u_i}v_i^*$.
Then, by evaluating at $x_i$, we have
\[
\begin{split}
\pi(T_i)&=\pi(w_i^*w_iu_iv_iv_i^*)\\
        &=\pi(q_iu_ip_i)\\
        &=\pi(q_{i-1}u_ip_i)\\
        &=\pi(q_{i-1}u_{i-1}p_i)\\
        &=\pi(q_{i-1}u_{i-1}p_{i-1})\\
        &=\pi(T_{i-1}).
\end{split}
\]
so that $T_{i}(x_i)-T_{i-1}(x_i)\in B$. Note that $T_i^*(x_i)T_i(x_i) \sim 1 $, $T_i(x_i)T_i^*(x_i) \sim 1$. From this,
\[ [T_i^*(x_i)T_i(x_i):T_{i-1}^*(x_i)T_{i-1}(x_i)]=[T_i(x_i)T_{i}^*(x_i):T_{i-1}(x_i)T_{i-1}^*(x_i)]\]
Subtracting  (\ref{E:(1)}) by (\ref{E:(2)}), we get
\[ \begin{split}
& [p_i(x_i):p_{i-1}(x_i)]-[q_i(x_i):q_{i-1}(x_i)]=[p_i(x_i):T_i^*(x_i)T_i(x_i)]-[q_i(x_i):T_i(x_i)T_i^*(x_i)] \\
&+[T_{i-1}^*(x_i)T_{i-1}(x_i):p_{i-1}(x_i)]- [T_{i-1}(x_i)T_{i-1}^*(x_i):q_{i-1}(x_i)]
\end{split}
\]
By Proposition \ref{P:index-essentialcodimension}\[k_i-l_i=t_i-t_{i-1}\] is obtained. We are going to show that this is the necessary and sufficient condition for Murray von Neumann equivalence of two projections in the corona algebra or the isomorphic equivalence of two homogeneous Hilbert module bundles modulo``finite'' sections.
\begin{prop}\label{P:von Neumann}
Suppose that two projections $\mathbf{p}, \mathbf{q}$ in the corona algebra of $C(X)\otimes B$ are represented by local liftings $(p_0,p_1,\dots,p_n)$, $(q_0,q_1,\dots,q_n)$ respectively where for each $i$ $p_i(x)$ and $q_i(x)$ are full, properly infinite projections for all $x \in X_i$. Let $k_i=[p_i(x_i):p_{i-1}(x_i)]$ and $l_i=[q_i(x_i):q_{i-1}(x_i)]$. Then $ \mathbf{p} \sim \mathbf{q}$ if and only if there are elements $t_i$'s in $K_0(B)$ such that
\begin{enumerate}
\item $k_i-l_i=t_i-t_{i-1}$
\item $t_i=0$ if $X_i$ has an infinite end point.
\end{enumerate}
\end{prop}
\begin{proof}
We already have shown one direction. For the other direction, suppose we are given $t_i$'s such that $k_i-l_i=t_i-t_{i-1}$.
Then for each $i$ we can produce a norm continuous projection valued function $r_i:X_i \to B$ such that $r_i \leq p_i$ and $[r_i]_0=t_i-t_{i-1}$ by Lemma \ref{L:subprojection}.  Let $p_i'=p_i-r_i$. Using Lemma \ref{L:properties} it follows that
$[p_i'(x_i):p_{i-1}(x_i)]=[p_i(x_i):p_{i-1}(x_i)]-[r_i]_0=k_i+t_{i-1}-t_i=l_i=[q_i(x_i):q_{i-1}(x_i)]$. We also check that $p_i'(x)$ are full, properly infinite projection for each $x\in X_i$ as in the proof of Lemma \ref{L:index-essentialcodimension}. Thus we can take a local lifting $(p_0',p_1',\dots,p_n')$ for $\mathbf{p}$ such that $k_i=l_i$. So we may assume that local liftings $(p_0,\dots,p_n)$ and $(q_0,\dots,q_n)$ for $\mathbf{p}$ and $\mathbf{q}$ satisfy the condition $k_i=l_i$.

Note that there is a path $u_i:X_i \to M(B)_s$ such that $u_i^*u_i=p_i$ and $u_iu_i^*=q_i$ since $p_i$ and $q_i$ are Murray von Neumann equivalent on $X_i$ for each $i$. Thus it is enough to show that $u_i(x_i)-u_{i-1}(x_i) \in B$. Note that $u_{i-1}(x)$ is a unitary from $p_{i-1}(x)H_B$ onto $q_{i-1}(x)H_B$ so that $\Ind(u_{i-1}(x_i))=0$.  Then $k_i=l_i$ implies that
    $$\Ind(q_i(x_i)u_{i-1}(x_i)p_i(x_i))=
   -l_i+\Ind(u_{i-1}(x_i))+k_i=0,$$
    where the first index is for maps from $p_i(x_i)H_B$ to $q_i(x_i)H_B$, and, for example, the index of $p_{i-1}(x_i)p_i(x_i)$ as a map from $p_i(x_i)H$ to $p_{i-1}(x_i)H$ is $k_i$.
    Also
    $$q_i(x_i)u_{i-1}(x_i)p_i(x_i)-u_{i-1}(x_i) \in B.$$  There is  a compact perturbation $v_i$ of $q_i(x_i)u_{i-1}(x_i)p_i(x_i)$ such that $v_i^*v_i=p_i(x_i)$, $v_iv_i^*=q_{i}(x_i)$, and $v_i-u_{i-1}(x_i)\in B$. Then we use the same argument in the proof of Proposition 3.14  to connect $v_i$ to $u_i(x)$ for some $x \in (x_i,x_{i+1})$. Then using this path we can arrange a strictly continuous function $w_i$ on $[x_i,x]$ such that $w_i^*w_i=p_i, w_iw_i^*=q_i$ and $w_i(x_i)-u_{i-1}(x_i)\in B$. Finally we paste two maps $w_i$ and $u_i$ at $x$ to generate $u_i'$.  In the $(-\infty,\infty)$-case we do the above for $i=1,\dots,n$ and let $u'_0=u_0$. In the circle case we do it for $i=0,\dots,n$.
\end{proof}

\begin{rem}\label{R:Cuntz}
\begin{itemize}
\item[(i)] In fact,  if there is no endpoints restriction then the condition i) in Proposition \ref{P:von Neumann} is automatically satisfied. Put $t_0$ any element in $K_0(B)$. Then let $t_i=k_i-l_i+t_{i-1}$ inductively. Thus $\mathbf{p} \sim \mathbf{q}$ is always satisfied so that the associated bundles are isomorphic. 
\item[(ii)] In the case of $X=(-\infty,\infty)$, we have the restrictions on $t_n=t_0=0$. But this implies that $\sum_{i=1}^n k_i- \sum_{i=1}^n l_i=t_n-t_0=0$ since the right hand side is the telescoping sum. Thus $\sum_{i=1}^n k_i=\sum_{i=1}^n l_i$. Then again $\mathbf{p} \sim \mathbf{q}$ is always satisfied by \cite[Proposition 3.14]{Lee2}.
\end{itemize} 
\end{rem}
It remains to obtain the conditions for homotopy equivalence of $\mathbf{p},\mathbf{q}$. Let us begin with some preparations. 

\begin{rem}
Without halving conditions, we need to put the condition that the essential codimension vanishes. In fact, for the projections $p,q$ in $B(H)$ such that $p-q \in K$ there is a unitary $u$ in $1+B(H)$ such that $upu^*=q$ if and only if $[p:q]=0$. Though we do not know whether this fact holds for any stable $C\sp*$-algebra $B$ and $\mt$, there is an affirmative answer for $\sigma$-unital purely infinite simple $C\sp*$-algebras such that $K_1(B)=0$ (see \cite[Theorem 2.14]{Lee}). And it has been crucial in KK-theory that the implementing unitary is of the form `identity + compact' \cite{Da,DE,DE1,Lee, Lee13}.
\end{rem}
We need the following lemma to prove the next theorem.
\begin{lem}
 Given an element $\mathbf{f} \in \mathcal{Q}(C(X)\otimes B)$ represented by  $(f_0,\cdots,f_n)$ we can find a representative $\widetilde{f} \in \mu$ of $\mathbf{f}$ such that $\widetilde{f}$ is continuous at each partition point.
 \end{lem}
\begin{proof}
 If $X$ is compact, then we set $x_0=0$, $x_{n+1}=1$. Otherwise, we set $x_0=x_1-1$ when $X$ contains $-\infty$, and $x_{n+1}=x_n +1$ when $X$ contaions $+\infty$.
Then we define a function in $C(X)\otimes B$ by
  \begin{equation*}
m_i(x)=
\begin{cases}
\frac{x-x_{i-1}}{x_i-x_{i-1}}(f_i(x_i) -f_{i-1}(x_i)), &\text{if $x_{i-1}\leq x \leq x_i$}\\
\frac{x-x_{i+1}}{x_i-x_{i+1}}(f_i(x_i) -f_{i-1}(x_i)), &\text{if $x_{i}\leq x \leq x_{i+1}$}\\
0,                               &\text{otherwise}
\end{cases}
\end{equation*}
for each $i=1,\cdots n$. In addition, we set $m_0=m_{n+1}=0$.
Then we define a function $\widetilde{f}$ from $f_i$'s by
\[\widetilde{f}(x)=f_i(x)-m_i(x)/2+m_{i+1}(x)/2\] on each $X_i$.
It follows that $f_i(x_i)-m_i(x_i)/2+m_{i+1}(x_i)/2=f_{i-1}(x_i)-m_{i-1}(x_i)/2+m_{i}(x_i)/2$. Thus $\widetilde{f}$ is well defined.
 \end{proof}
 \begin{prop}
 Suppose that two projections $\mathbf{p}, \mathbf{q}$ in the corona algebra of $C(X)\otimes B$ are represented by local liftings $(p_0,p_1,\dots,p_n)$, $(q_0,q_1,\dots,q_n)$ respectively where for each $i$ $p_i(x)$ and $q_i(x)$ are halving projections for all $x \in X_i$. Let $k_i=[p_i(x_i):p_{i-1}(x_i)]$ and $l_i=[q_i(x_i):q_{i-1}(x_i)]$. Then $ \mathbf{p} \sim_u \mathbf{q}$ if and only if there are elements $t_i$'s and $s_i$'s in $K_0(B)$ such that
\begin{enumerate}
\item $k_i-l_i=t_i-t_{i-1}$,
\item $s_i+t_i=s_{i-1}+t_{i-1}$,
\item $t_i$ and $s_i$ are zero if $X_i$ has an infinite end point.
\end{enumerate}
 \end{prop}
 \begin{proof}
 For one direction, suppose $\mathbf{p} \sim_u \mathbf{q}$. Then $\mathbf{p} \sim \mathbf{q}$ and $\mathbf{1-p} \sim \mathbf{1-q}$. By applying Proposition \ref{P:von Neumann} to $\mathbf{p},\mathbf{q}$, we obtain $t_i$'s in $K_0(B)$ such that $k_i-l_i=t_i-t_{i-1}$. Similarly, by applying Proposition \ref{P:von Neumann} to $\mathbf{1-p}, \mathbf{1-q}$, we obtain $s_i$'s such that $s_i-s_{i-1}=l_i-k_i$ since $[1-p_i(x_i):1-p_{i-1}(x_i)]=-k_i$ and  $[1-q_i(x_i):1-q_{i-1}(x_i)]=-l_i$.
   
 Now given $t_i$'s and $s_i$'s in the assumption using the same argument in Proposition \ref{P:von Neumann} we can construct local liftings $\mathbf{v}=(v_0,\dots,v_n)$, $\mathbf{w}=(w_0,\dots,w_n)$ such that $v_i^*v_i=p_i$, $v_iv_i^*=q_i$, $w_i^*w_i=1-p_i$, $w_iw_i^*=1-q_i$. Now consider $\mathbf{u}=\mathbf{v}+\mathbf{w}$.     
 \end{proof}

 \begin{rem}
\begin{enumerate}
  \item Again, if there is no endpoint restrictions, we can construct $t_i$'s inductively so that the conditions $k_i-l_i=t_i-t_{i-1}$ are satisfied. Similary $s_i$'s such that $l_i-k_i=s_i-s_{i-1}$. Hence any two homogeneous projections are unitary equivalent.
 \item In the case of $X=(-\infty, \infty)$, we have the restrictions on $s_n=t_n=s_0=t_0=0$. It follows that $\sum k_i=\sum l_i$. Then $\mathbf{p} \sim \mathbf{q}$ and $1-\mathbf{p} \sim 1-\mathbf{q}$ by \cite[Corollary 3.15]{Lee2} since $p_i(x), q_i(x)$ are halving projections for each $x\in X_i$.
\end{enumerate}   
 \end{rem}
A (non-unital) $C\sp*$-algebra $B$ satisfies the good index theory if whenever $B$ is embedded as an ideal of $E$ and $u$ is a unitary of $E/B$ such that $\partial_1([u]_1)=0$, where $\partial_1:K_1[(E/B)\to K_0(B)$, then there is a unitary in $E$ which lifts $u$. 
\begin{prop}\label{P:homotopy}
Suppose that $C(X)\otimes B$ satisfies the good index theory. Under the same assumtions on $\mathbf{p}$ and $\mathbf{q}$, $\mathbf{p} \sim_{h} \mathbf{q}$ if and only if 
\begin{enumerate}
\item $k_i-l_i=t_i-t_{i-1}$,
\item $s_i+t_i=s_{i-1}+t_{i-1}=0$,
\item $t_i$ and $s_i$ are zero if $X_i$ has an infinite end point.
\end{enumerate}
\end{prop}
\begin{proof}
If $\mathbf{u}$ is the implementing unitary such that $\mathbf{u}\mathbf{p}\mathbf{u}^*=\mathbf{q}$, then the good index theory implies that there is a unitary $u $ in the the pathcomponet of $1_{\mu}$ such that $\pi(u)=\mathbf{u}$. Recall that $t_i$ is the index of $u_i=q_iu|_{X_i}p_i$ or the index of $T_i$ which is constant throughout $X_i$ where $\pi(u^*u)=\mathbf{p}$ and $\pi(uu^*)=\mathbf{q}$. Similarly, $s_i$ is the index of $v_i=(1-q_i)v|_{X_i}(1-p_i)$ where $\pi(v^*v)=\mathbf{1-p}$ and $\pi(vv^*)=\mathbf{1-q}$. So if we take $u_i=u|_{X_i}p_i$ and $v_i=u|_{X_i}(1-p_i)$, then
\[
\begin{split}
t_i+s_i=&\Ind(q_iu_ip_i)+\Ind((1-q_i)u_i(1-p_i))\\
           =&\Ind(u|_{X_i})=0
\end{split}
\]
For the converse direction, we may assume that $K_1(\mathcal{Q}(C(X)\otimes B))$ does not vanish. Then as in the proof of Proposition \ref{P:von Neumann}, we can construct $\mathbf{v}=(v_0,u_v,\dots,v_n)$, $\mathbf{w}=(w_0,w_1,\dots,w_n)$ such that $\mathbf{v}\mathbf{v}^*=\mathbf{q}, \mathbf{v}^*\mathbf{v}=\mathbf{p}$ and $\mathbf{w}^*\mathbf{w}=\mathbf{1-p}$, $\mathbf{w}\mathbf{w}^*=\mathbf{1-q}$. Note that $v_i$ and $w_i$ define the indices $t_i$ and $s_i$ respectively. Thus $\mathbf{u}=\mathbf{v}+\mathbf{w}$ is a unitary and $[\mathbf{u}]_1=\Ind(v_i+w_i)=t_i+s_i=0$ via the map $K_1(\mathcal{Q}(C(X)\otimes B)) \to K_0(C(X)\otimes B) \to K_0(B)$, where the latter is from evaluation at a point. Thus the good index theory implies that there is a unitary $u \in \mu$ which lifts $\mathbf{u}$. This fact and the pathconnectedness of $\mu$ imply that $\mathbf{u}$ is also connected to the identity.  
\end{proof}

The following corollary can be obtained  from \cite[Corollary 4.13]{BL}. We recapture it as an illustration of how to use the above proposition since the proof cannot be copied.
\begin{cor}
Let $B=K$ be the algebra of compact operators on a separable Hilbert space. Let $\mathbf{p}$, $\mathbf{q}$ define infinite dimensional Hilbert space bundles. Then $\mathbf{p} \sim_{h} \mathbf{q}$ if and only if 
\begin{enumerate}
\item $k_i-l_i=t_i-t_{i-1}$,
\item $s_i+t_i=s_{i-1}+t_{i-1}=0$,
\item $t_i$ and $s_i$ are zero if $X_i$ has an infinite end point.
\end{enumerate} 
\end{cor}
\begin{proof}
It is well known that $C(X)\otimes B$ is of stable rank one and such an algebra satisfies the good index theory. Thus the conclusion follows from Proposition \ref{P:homotopy}. 
\end{proof}
\begin{cor}
 Suppose that $X$ does not contain any infinite point and  $C(X)\otimes B$ satisfies the good index theory. Any two homogeneous projections  $\mathbf{p}$ and $\mathbf{q}$ in the corona algebra of $C(X)\otimes B$ are homotopic.  
\end{cor}
\begin{proof}
The proof is similar to one given in Remark \ref{R:Cuntz}. Given  sequences $\{k_i\}_{i=1}^{n}$, $\{l_i\}_{i=1}^n$, choose any $t_0$ in $K_0(B)$, then we take $s_0=-t_0$, $t_1=k_1-l_1+t_0$. Inductively we can repeat this procedure so that we can arrange sequnecs $\{t_i\}_{i=0}^n$, $\{s_i\}_{i=0}^n$ such that 
\[k_i-l_i=t_i-t_{i-1}, \quad s_i+t_i=s_{i-1}+t_{i-1}=0.\] By Proposition \ref{P:homotopy}, the conclusion holds.  
\end{proof}
 
\end{document}